\documentclass[12pt,reqno]{article}

\title{An isoperimetric inequality for planar triangulations}
\author{Omer Angel \and Itai Benjamini \and Nizan Horesh}

\usepackage{amsmath,amssymb,amsfonts,amsthm}

\usepackage[colorlinks=true,linkcolor=black,citecolor=black,urlcolor=blue,pdfborder={0 0 0}]{hyperref}

\usepackage{graphicx}
\usepackage{enumitem}
\usepackage{color}

\usepackage[font=sf, labelfont={sf,bf}, margin=1cm]{caption}

\usepackage[nameinlink]{cleveref}
\crefname{thm}{Theorem}{Theorems}

\newtheorem{thm}{Theorem}
\newtheorem{prop}[thm]{Proposition}
\newtheorem{coro}[thm]{Corollary}
\newtheorem{problem}{Problem}
\numberwithin{equation}{section}

\newcommand{\R}{\mathbb R}
\newcommand{\bd}{\partial}
\newcommand{\eps}{\varepsilon}

\begin{document}

\maketitle

\begin{abstract}
  We prove a discrete analogue to a classical isoperimetric theorem of Weil for
  surfaces with non-positive curvature. It is shown that hexagons in the
  triangular lattice have maximal volume among all sets of a given boundary
  in any triangulation with minimal degree $6$.
\end{abstract}

\section{Introduction}

In 1926 Weil \cite{W, Z} proved the following:

\begin{thm}[Weil]
  If $M$ is a $2$-dimensional manifold homeomorphic to the unit disc with
  non-positive curvature at every point, then
  \[
    |\bd M| \geq 2\sqrt{\pi |M|} .
  \]
\end{thm}

Thus a disc in the Euclidean plane solves the isoperimetric problem, not
just among sets in the plane but also among all surfaces with non-positive
curvature.  We give a discretized version of this:

\begin{thm}\label{thm:isoV}
  Any disc triangulation with $V$ vertices and $n$ boundary vertices, and
  with all internal degrees at least $6$ has
  $V\leq \left \lfloor \frac{(n+3)^2}{12} \right\rfloor$.  Equality is
  achieved in the Euclidean triangular lattice.
\end{thm}

We remark that unlike the proofs in \cite{W, Z}, our arguments are purely
combinatorial and do not use conformal geometry.  We discuss below several
consequences and variants of the argument, including for hyperbolic
triangulations.

\section{Proof}

% \section{proof}

\begin{proof}[Proof of Theorem 1]
  If $V\le 6$ then there can be no internal vertices, and the claim clearly
  holds.  We assume henceforth that $V\ge 6$.  Let $E,F$ be the number of
  edges and faces in the triangulation.  By Euler's formula, $V-E+F=1$
  (since the external face is not counted).  Let $n=|\bd T|$.  Counting
  faces with a marked edge gives $3F=2E-n$, and thus $E = 3V-n-3$.

  Let $\sigma\ge 6$ be the average degree of internal vertices, and let $M$
  be the set of internal edges, incident to the boundary, with a marked
  endpoint on the boundary, and $m=|M|$ its cardinality (so that an
  internal edge between two boundary vertices is counted twice).  Summing
  vertex degrees gives
  \[
    2E = \sigma (V-n) + 2n + m,
  \]
  which gives
  \[
    m = 2n-6-(\sigma-6)(V-n),
  \]
  and since $\sigma\ge 6$ and $V-n$ is the number of internal vertices,
  $m\le 2n-6$.

  Stripping all boundary vertices and all faces incident to the boundary
  leaves a smaller triangulation, which need not be a disc triangulation:
  It may have several connected components, and may have components with a
  non-simple boundary, containing a cut vertex, or even bridges (See
  \cref{fig:strip}).  Let $n'$ be the number of boundary edges of the
  stripped triangulation (with bridges counted twice).  Our objective is to
  show that $n'\leq n-6$.

  \begin{figure}[h]
    \centering
    \includegraphics[width=.9\textwidth]{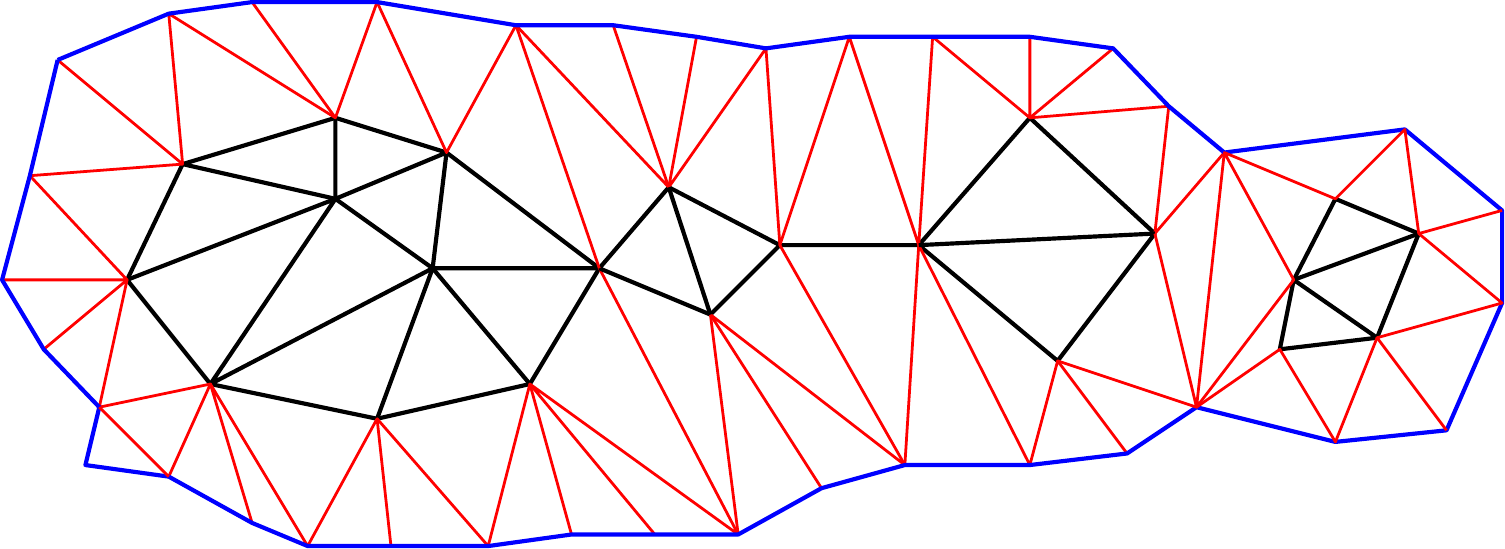}
    \caption{A disc triangulation with a boundary of length $37$.
      Stripping faces adjacent to the boundary leaves a smaller
      triangulation, with total boundary length $23\leq 37-6$.}
    \label{fig:strip}
  \end{figure}

  We now count edges of $M$ with a marked side (a face incident to them).
  The total is $2m$, and we count it by considering the different types of
  stripped faces.   Each edge of the new boundary
  (of which there are $n'$) must have a third vertex on the outer boundary,
  and contains two edges of $M$.  All other stripped faces have two or
  three vertices on the outer boundary.  Let $\alpha$ be the number of
  faces with a boundary edge and internal vertex.  Let $\gamma$ be the
  number of faces with two boundary vertices and one internal vertex, but
  no boundary edges.  Let $\beta_i$ be the number of faces with three
  boundary vertices and $i$ boundary edges (see \cref{fig:face_types}).
  Then we have the following identity:
  \[
    2m = 2n' + 2\alpha + 4\gamma + 6\beta_0 + 4\beta_1 + 2\beta_2,
  \]
  as well as
  \[
    n = \alpha + \beta_1 + 2\beta_2 + 3\beta_3.
  \]

  \begin{figure}
    \centering
    \includegraphics[width=.9\textwidth]{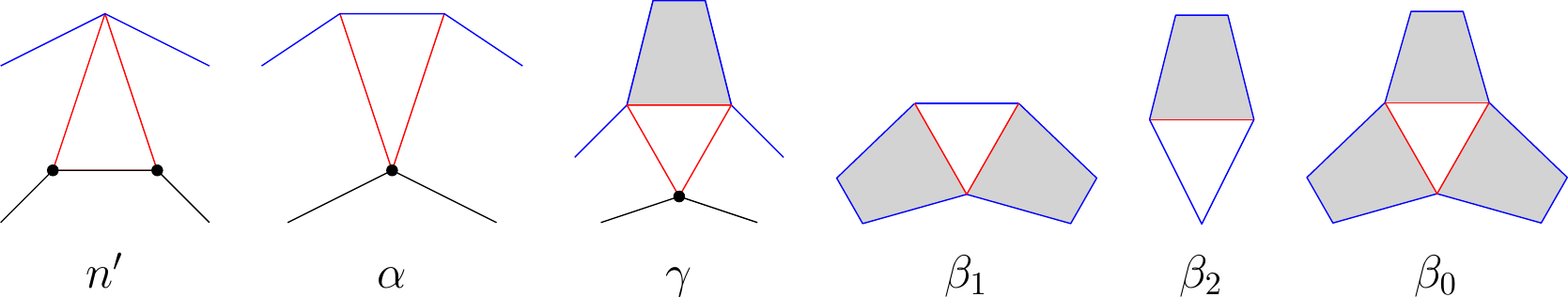}
    \caption{Different types of faces being stripped, and the number of
      such faces.  Internal vertices are marked.  The shade areas may
      contain any number of faces and vertices.}
    \label{fig:face_types}
  \end{figure}

  Combining these identities with $m\leq 2n-6$ gives
  \[
    n' + 2\gamma + 3\beta_0 + \beta_1 - \beta_2 - 3\beta_3 \leq n-6.
  \]
  Observe that if $\beta_3$ is non-zero, the entire triangulation consists
  of a single face, a case we already dealt with.  Otherwise, we need to
  show $\beta_2 \leq \beta_1 + 3\beta_0 + 2\gamma$.

  A face contributing to $\beta_2$ has a unique internal edge with both
  endpoints on the boundary.  The face $f$ on the other side of this edge will
  contribute to $\beta_1,\beta_3$, or $\gamma$.

  If $f$ has an internal third vertex, it counts towards $\gamma$.
  Alternatively, $f$ may have its third vertex on the boundary.  If $f$ is
  of type $\beta_2$, then the entire triangulation has $4$ vertices, all on
  the boundary.  If this face is of type $\beta_0$ or $\beta_1$ then it
  also contributes to $\beta_1 + 3\beta_0 + 2\gamma$.  A face of type
  $\beta_0$ can correspond to at most $3$ faces of type $\beta_2$.  It is
  possible for a face of type $\beta_1$ to correspond to two faces of type
  $\beta_2$, but this only happens if the entire triangulation has $5$
  vertices, a case we already considered.  Thus $n'\leq n-6$ as claimed.

  Separating the stripped triangulation into a collection of disc
  triangulations with total boundary size $n'$, and proceeding by
  induction, we find that at distance $k$ from the boundary there are at
  most $n-6k$ vertices. and summing this gives the volume of a Euclidean
  hexagon as a bound on $|T|$.   It follows that the overall number of
  vertices is bounded by
  \[
    V \leq \sum_{k\geq 0} (n-6k)^+
    = \left\lfloor \frac{(n+3)^2}{12}\right\rfloor.
  \]
  (The last identity is easily verified by checking cases for $n\mod 6$).

  \medskip

  In the case of a convex hexagon in the triangular lattice we indeed have
  $n' = n-6$, unless the hexagon has width 1, in which case $m=0$.  It
  follows that equality is achieved by convex hexagons in the triangular
  lattice with all boundary segments of suitable lengths in
  $\{k,k+1,k+2\}$.  For example, if $n=6k+7$ the boundary segments have
  lengths $(k,k+2,k+1,k+1,k+1,k+2)$ in order.
\end{proof}

Since any map with no multiple edges can be made into a triangulation by
adding edges, it follows that the same inequality holds for any simple map
in a disc.

\begin{coro}
  Any simple map in a disc with $V$ vertices and $n$ boundary vertices, and with
  all internal degrees at least $6$ has
  $V\leq \left \lfloor \frac{(n+3)^2}{12} \right\rfloor$.
\end{coro}

The edge boundary is also minimized by hexagons in the triangular lattice.

\begin{thm}
  If $T$ is a triangulation where all vertices have degree at least $6$,
  and $A$ is any finite set in $T$ of size $V$, then the edge boundary $\bd
  A$ satisfies $|\bd A| \geq \sqrt{48V}$.
\end{thm}

\begin{proof}
  Let $n$ be the number of boundary vertices in $A$.  From \cref{thm:isoV}
  we have that $12V \leq (n+3)^2$.  The number of directed edges from these
  $n$ vertices into the triangulation is at most $2n-6$ (as in the proof
  above). Since there are at most $2n$ directed edges in the boundary (less
  if it contains any bridges) we have that the number of outward edges is
  at least $2n+6 \geq \sqrt{48 V}$.
\end{proof}

\section{Hyperbolic lattices}

In \cite{HJL} it is proved that in the hyperbolic lattice balls solve the
Isoperimetric problem.  This involves using $m = 2n-6-(\sigma-6)(V-n)$ with
a better bound on the last term.  Let $V_R$ denote the number of vertices
in a ball of radius $R$ in the $\delta$-regular triangulation, and let
$S_R$ denote the size of its boundary.  (These grow exponentially; Explicit
formulae are available but not helpful for us.)

We have the following result, in a sense stating that hyperbolic balls are
optimal among all triangulations with minimal degree $\delta$.

\begin{thm}
  If $T$ is a disc triangulation with minimal degree $\delta$ and $V \geq
  V_R$ vertices.  Then $n\geq S_R$.
\end{thm}

\begin{proof}
  Repeating the argument above when all vertex degrees are at least some
  $\delta>6$, we find that after stripping a layer, the new boundary $n'$
  satisfies
  \[
    n' \leq n-6 - (\delta-6)(V-n).
  \]
  In the case of a ball in the $\delta$-uniform triangulation, there are no
  faces of type $\beta_i$, and there is equality here.

  Suppose for a contradiction that there exists a triangulation with
  $V\geq V_R$ and $n<S_R$, and consider the example with minimal possible
  $R$f.  Then $V-n > V_R - S_R = V_{R-1}$, and
  $n' < S_R-6-(\delta-6)(V_R-S_R) = S_{R-1}$.  Thus stripping a layer gives
  a smaller counter-example.
\end{proof}

\begin{problem}
  Is it true that for any triangulation with minimal degree $\delta$, there
  exists a subset of the $\delta$-regular triangulation with the same
  volume and equal or smaller boundary?
\end{problem}

Finally, we remark that hyperbolicity follows once there are enough
vertices of degree greater than 6.

\begin{prop}
  For every $R>0$ there exists $\alpha>0$ such that the following holds.
  Let $T$ be an infinite plane triangulation with minimal degree $6$, and
  such that for some $R$, every ball of radius $R$ contains a vertex of
  degree at least $7$.  Then $T$ is non-amenable, with isoperimetric
  constant at least $\alpha$.
\end{prop}

\begin{proof}
  First, we may assume that $T$ has bounded degrees.  Indeed, we may replace
  each vertex of degree greater than $9$ by several vertices of degree at
  least $7$, while keeping all other degrees at least $6$.  This maintains
  or reduces the distance to a high degree vertex from any point.
  
  Consider a finite set $S\subset T$ of size $V$.  To optimize the
  isoperimetric ratio, we may assume $S$ is connected.  Let $\sigma$ be the
  average degree of internal vertices in $S$.  Then we have as above
  $m = 2n - 6 - (\sigma - 6)(V - n)$.

  For some $\eps>0$ to be fixed later, we argue as follows.  If
  $\sigma\geq 6+\eps$ this implies the boundary is proportional to $V$.  If
  $\sigma<6-\eps$, then there are at most $\eps V$ vertices of degree
  greater than $6$.  Since degrees are bounded by $9$, each of these is
  within distance $R$ of at most $9^R\eps$ vertices, and so $(1-9^R\eps)V$
  vertices are within distance $R$ of $|\partial S|$.  This implies
  $|\partial S| \geq c V$ for some $c$, provided $9^R \eps < 1$.
\end{proof}

It should be possible to relax the condition of minimal degree $6$ to the
following.  Suppose there is a partition of the vertices into sets $A_i$
such that each $A_i$ has diameter at most $R$ and such that within each
$A_i$ the average degree is greater than $6$.  

We remark that the proof above gives $\alpha > e^{-c/\eps}$.  The proof can
be adapted to have $\alpha$ polynomial in $R$.

\begin{problem}
  For any given $R$, what is the optimal isoperimetric constant for a
  triangulation as above?
\end{problem}

\section{Further questions}

Kleiner (1992) \cite{K} proved an analogue of Weil's Theorem for three
dimensional manifolds of non-positive curvature, and Croke (1984) \cite{C}
proved a four dimensional analogue.  The question in all higher dimensions
is still open.

%{\em Can the technique of proof be adapted to higher dimensions?}

It would be interesting to have a three (or higher) dimensional version of
our theorems.  A specific problem in three (or more) dimensions is as follows.
% arose recently in a discussion with {\em Ivan Izmestiev}.

\begin{problem}
  Show that for every CAT(0) cubulation of a ball with $V$ cubes and $A$
  boundary squares there is a subset of cubes in the cubic lattice of
  $\R^3$ of size $V$ and at most $A$ boundary squares?
\end{problem}

Here, a cubulation of a ball is a collection of topological cubes, glued
along faces to get a topological ball.  There are various possible
analogues to the having minimal degree $6$.  Being CAT(0) is a well studied
property of general metric spaces, analogous to having non-positive
curvature.  The Cartan-Hadamard Theorem (see e.g.\
\cite[Theorem~II.4.1]{BH}) implies that a plane triangulation is CAT(0) if
and only if it has minimal degree at least $6$.

It is similarly possible to characterize a CAT(0) cubulation in terms of
the possible local structures at each vertex, and in particular it is
necessary (though not sufficient) that each vertex to be in at least $8$
cubes.  We refer the reader to \cite{BH} for the theory of CAT(0) metric
spaces, and in particular Theorems~II.5.2 and II.5.18 for the
characterization in terms of local structure.

\bigskip

Finally, it is also possible for the regular triangulations to be extremal
in other ways.

\begin{problem}
  If $T$ is a plane triangulation with all vertex degrees at least $6$, is
  it necessarily true that the connective constant satisfies
  $\mu(T) \geq \mu(\mathbb{T}_6)$? Is it true that
  $p_c^{\mathrm{site}}(T) \leq 1/2$, and
  $p_c^{\mathrm{bond}}(T) \leq 2\sin(\pi/18)$?
\end{problem}

$1/2$ and $2\sin(\pi/18)$ are the values for the 6-regular triangular
lattice $\mathbb{T}_6$ \cite{SE}.  Similarly, for other $k$, do the
$k$-regular triangulation have the minimal connective constant and maximal
percolation threshold.

% \paragraph{Acknowledgment.}
% % This note was initiated when OA and NH were students at Weizmann.
% OA is supported in part by NSERC.

\bigskip

\begin{enumerate}[font=\sc, leftmargin=40mm, labelsep=5mm]

\item[Omer Angel]
  University of British Columbia \\
  \verb+angel@math.ubc.ca+

\item[Itai Benjamini]
  Weizmann Institute \\
  \verb+itai.benjamini@gmail.com+

\item[Nizan Horesh]
  Intel \\
  \verb+nizanh@gmail.com+

\end{enumerate}

\end{document}